\documentclass[11pt]{article}
\usepackage{amsmath, amsfonts,amsthm,amssymb,amsbsy,upref,color,graphicx,amscd,hyperref,makeidx,enumerate}
\usepackage{tikz}
\usepackage[active]{srcltx}
\usepackage[latin1]{inputenc}
\usepackage{enumerate}
\usepackage[english]{babel}
\usepackage{newlfont}
\usepackage{amsbsy}
\usepackage[english]{babel}
\usepackage[center]{caption2}
\usepackage{amssymb}
\usepackage{amsmath}
\usepackage{latexsym}
\usepackage{amsthm}
\usepackage{amsthm}
\theoremstyle{plain}
\newtheorem{thm}{Theorem}

\newtheorem{lem}[thm]{Lemma}

\newtheorem{nota}[thm]{Notation}
\newtheorem{rem}[thm]{Remark}
\newtheorem{defin}[thm]{Definition}

\newcommand{\R}{\mathbb{R}}

\newcommand{\N}{\mathbb{N}}

\usepackage{mathtools}
\def\multiset#1#2{\ensuremath{\left(\kern-.2em\left(\genfrac{}{}{0pt}{}{#1}{#2}\right)\kern-.2em\right)}}
\usepackage[center]{caption2}

\begin{document}

\title{Weighted graphs with distances in given ranges }
\author{ Elena Rubei}
\date{}
\maketitle

\def\thefootnote{}
\footnotetext{ \hspace*{-0.52cm}
{\bf 2010 Mathematical Subject Classification: 05C05, 05C12, 05C22} 

{\bf Key words: weighted graphs,  distances, range} }

\begin{abstract}
Let  ${\cal G}=(G,w)$ be a weighted simple finite connected  graph,
that is, let $G$ be a simple finite  connected
 graph endowed with a function $w$
from the set of the edges of $G$ to the set of real numbers.
For any subgraph $G'$ of  $G$, we define  $w(G')$ to be the sum of the weights of the edges of $G'$. 
For any  $i,j $ vertices of $G$, we define $D_{\{i,j\}} ({\cal G})$ to be
the minimum of the weights of the simple paths of $G$ joining $i$ and $j$. The  $D_{\{i,j\}} ({\cal G})$ are called $2$-weights of  ${\cal G}$. 
Weighted graphs and their reconstruction from $2$-weights have applications in several disciplines, such as 
biology and psychology.

Let $\{m_I\}_{I  \in {\{1,...,n\} \choose 2}}$ and
$\{M_I\}_{I  \in {\{1,...,n\} \choose 2}}$ be
two families of positive real numbers parametrized by the $2$-subsets of 
$ \{1,..., n\}$ with $m_I \leq M_I$ for any $I$;  we study when 
there exist a positive-weighted graph ${\cal G}$
and an  $n$-subset $\{1,..., n\}$ of  the set of its vertices such that 
$D_I ({\cal G}) \in [m_I, M_I] $ for any $I  \in {\{1,...,n\} \choose 2}$.
Then we study the analogous problem for trees, both in the case of positive weights and in the case of general weights.
\end{abstract}

\section{Introduction}

For any graph $G$, let $E(G)$, $V(G)$ and $L(G)$ 
 be respectively the set of the edges,   
the set of the vertices and  the set of the leaves of $G$.
A {\bf weighted graph} ${\cal G}=(G,w)$ is a graph $G$ 
endowed with a function $w: E(G) \rightarrow \R$. 
For any edge $e$, the real number $w(e)$ is called the weight of the edge. If 
all the weights are nonnegative (respectively positive), 
we say that the graph is {\bf 
nonnegative-weighted} (respectively {\bf positive-weighted}).
Throughout the paper we will consider only simple finite  connected graphs.

For any subgraph $G'$ of  $G$, we define  $w(G')$ to be the sum of the weights of the edges of $G'$. 

\begin{defin}
Let ${\cal G}=(G,w) $ be a weighted graph. 
For any distinct $ i,j \in V(G)$, we define $$ D_{\{i,j\}}({\cal G}) = min 
\{w(p) | \; p \text{ a simple path of } G  \text{ joining }  i  \text{ and } j\}.$$ More simply, we denote $D_{\{i,j\}}({\cal G})$ by
$D_{i,j}({\cal G})$ for any order of $i,j$.  
We call  the  $ D_{i,j}({\cal G})$ the {\bf $2$-weights} (or distances) of ${\cal G}$.
\end{defin}

Observe that in the case ${\cal G}$ is a tree,  $ D_{i,j}({\cal G})$ is  the weight of the unique path joining $i$ and $j$.

 If $S $ is a subset of $V(G)$, the $2$-weights give
a vector in $\mathbb{R}^{  S \choose 2}$. This vector is called 
$2${\bf -dissimilarity vector} of $({\cal G}, S)$.
Equivalently,  we can speak 
of the {\bf family of the $2$-weights}  of $({\cal G}, S)$.

We can wonder when a family of real numbers is the family of the $2$-weights of some weighted  graph and of some subset of the set of its vertices.
If $S$ is a finite set of cardinality greater than $2$,
we say that a family of  real numbers
 $\{D_{I}\}_{I \in {S \choose 2}}$   is {\bf graphlike} (respectively p-graphlike, nn-graphlike) if 
there exist a weighted graph (respectively a positive-weighted graph, a nonnegative-weighted  graph) ${\cal G}=(G,w)$ and a subset $S$ 
of the set of its vertices such that $ D_{I}({\cal G}) = D_{I}$  for any 
 $2$-subset $I$ of $ S$.
If the graph is a weighted (respectively positive-weighted, 
 nonnegative-weighted)  tree
 ${\cal T}=(T,w)$
we say that the family 
 is {\bf treelike} (respectively  p-treelike, nn-treelike). 
If, in addition, $S \subset L(T)$, we say that the family
 is  {\bf l-treelike} (respectively, p-l-treelike, nn-l-treelike).

Weighted graphs have applications in several disciplines, such as 
biology and psychology. Phylogenetic trees are weighted graphs
whose vertices represent  species and the weight of an edge is given 
by how much the DNA sequences of the species represented by the vertices of the edge differ. Dissimilarity families arise naturally also in 
 psychology, see for instance the introduction in \cite{C-S}.
There is a wide literature concerning graphlike dissimilarity families
and treelike dissimilarity families,
in particular concerning methods to reconstruct weighted trees from their dissimilarity families; these methods are used by biologists to reconstruct phylogenetic trees. 
See for example \cite{N-S}, \cite{S-K} and  \cite{Dresslibro}, \cite{S-S} for overviews.

The first contribution to the 
characterization of  graphlike families of numbers dates back  to 1965 and it is due  to Hakimi and Yau, see  \cite{H-Y}:

\begin{thm} \label{Hakimi-Yau} {\bf (Hakimi-Yau)}
A family of positive real numbers
$\{D_{I}\}_{I \in {\{1,...,n\} \choose 2}}$
is p-graphlike if and only if the $D_I$ satisfy the 
triangle inequalities, i.e. if and only if $D_{i,j} \leq D_{i,k} + D_{k,j}$ for any distinct $i,j,k \in [n]$.
\end{thm}

In the same years, also a criterion for a metric on
a finite set to be nn-l-treelike
was established, see \cite{B}, \cite{SimP}, \cite{Za}:

\begin{thm} \label{Bune} {\bf (Buneman-SimoesPereira-Zaretskii)}
Let $\{D_{I}\}_{I \in {\{1,...,n\} \choose 2}}$ be a set of positive real numbers 
satisfying the triangle inequalities.
It is p-treelike (or nn-l-treelike) 
 if and only if, for all distinct $i,j,k,h  \in \{1,...,n\}$,
the maximum of $$\{D_{i,j} + D_{k,h},D_{i,k} + D_{j,h},D_{i,h} + D_{k,j}
 \}$$ is attained at least twice. 
\end{thm}


Also   the case of not necessarily nonnegative weights has been studied.
In 1972 Hakimi and Patrinos proved the following theorem 
(see \cite{H-P}):

  \begin{thm} {\bf (Hakimi-Patrinos)}
 A   family of real numbers $\{D_{I}\}_{I \in {\{1,...,n\} \choose 2}}$
  is always the family of the $2$-weights
of some weighted graph and some subset $\{1,...., n\}$ of its vertices 
\end{thm}

In \cite{B-S}, Bandelt and Steel proved a result, analogous to
Theorem \ref{Bune}, for general weighted trees:

\begin{thm} \label{Bandelt-Steel} {\bf  (Bandelt-Steel)}
For any set of real numbers $\{D_{I}\}_{I \in 
{\{1,...,n\} \choose 2}}$,
 there exists a weighted tree ${\cal T}$ with leaves $1,...,n$
such that $ D_{I} ({\cal T})= D_{I}$  for any  $2$-subset $I$ of
 $\{1,...,n\}$  if and only
 if, for any distinct $a,b,c,d \in  \{1,...,n\}$,  we have that at least two among 
 $$ D_{a,b} + D_{c,d},\;\;D_{a,c} + D_{b,d},\;\; D_{a,d} + D_{b,c}$$
are equal.
\end{thm}

Recently Baldisserri characterized the families $\{D_{I}\}_{I \in {\{1,...,n\} \choose 2}}$ that are the families of the  $2$-weights of
positive-weighted trees with exactly $n$ vertices, see \cite{Ba}.

Finally we want to mention that recently 
$k$-weights of weighted graphs for $k \geq 3$ have been introduced and studied; in particular there are some results concerning the characterization of families of $k$-weights,
see for instance \cite{B-C}, \cite{B-R}, \cite{B-R2},   \cite{H-H-M-S}, \cite{Iri}, 
\cite{L-Y-P}, \cite{P-S}, \cite{Ru1}, and \cite{Ru2}.

In this paper, we study when there exists a weighted graph with 
$2$-weights  in given ranges;
 this problem can be of interest because the data one can get 
 from experiments  are obviously  not precise, on the contrary 
 they can vary in a range.
Precisely,
let $\{m_I\}_{I  \in {\{1,...,n\} \choose 2}}$ and
$\{M_I\}_{I  \in {\{1,...,n\} \choose 2}}$ be
two families of positive real numbers parametrized by the $2$-subsets of 
$ \{1,..., n\}$ with $m_I \leq M_I$ for any $I$;  in \S3 we study when 
there exist a weighted graph ${\cal G}$
and an  $n$-subset $\{1,..., n\}$ of  the set of its vertices such that 
$D_I ({\cal G}) \in [m_I, M_I] $ for any $I  \in {\{1,...,n\} \choose 2}$.
Finally, in \S4 we study the analogous problem for trees, both in the case of positive weights and in the case of general weights.
The treatment of the case of trees turns out to be much more complicated and long than the case of graphs.


\section{Preliminaries}

\begin{nota} \label{notainiziali}


$\bullet$  For any $n \in \N-\{0\}$, let $[n]= \{1,..., n\}$.

$\bullet$  For any set $S$ and $k \in \mathbb{N}$,  let ${S \choose k}$
be the set of the $k$-subsets of $S$.

$\bullet$ For any family of real numbers or unknowns parametrized  by  ${[n]  \choose 2}$,  $\{x_{\{i,j\}}\}_{\{i,j\} \in {[n] \choose 2 }}$,  we denote $x_{\{i,j\}}$ by $x_{i,j}$ for any order of $i$ and $j$.

$\bullet$ Throughout the paper, the word ``graph'' will denote a finite simple connected graph.

$\bullet$ Let $T$ be a tree and let $S$ be a subset of $L(T)$. We denote by $T|_S$ the minimal subtree 
of $T$ whose set of vertices  contains $S$.

$\bullet$ Let $T$ be a tree.
We say that two leaves  $i$ and $j$ of $T$ are neighbours
if in the path  joining $i$ and $j$ there is only one vertex of degree greater than or equal to $3$. 

\end{nota}

The following theorem (see \cite{Ca}) and the following lemma will be useful to solve our problem in the case of trees.

\begin{thm} {\bf (Carver)}  Let $L_i (x_1,....,x_t) $ for $i=1,....,s$ be  polynomials 
of degree $1$ in $x_1,....., x_t$. The system of inequalities
$$ \left\{ \begin{array}{l} 
L_1 (x_1,...., x_t) >0  \\
...... \\...... \\
L_s (x_1,...., x_t) >0 
\end{array} \right.$$
is solvable if and only if there does not exist a set of $ s+1 $ constants, $c_1,....., 
c_{s+1}$, such that $$  \sum_{i=1,....,s} c_i L_i (x_1,...., x_t) + c_{s+1} \equiv 0, $$
at least one of the $c$'s being positive and none of them being negative.
\end{thm}

\begin{rem}  \label{eliminating} Let $S$ be a system of linear inequalities in $x_1,...., x_t$. We can write it as follows:
$$  \left\{ \begin{array}{l}
x_t > L_1 (x_1,...., x_{t-1}) \\
.......\\
.......\\
x_t > L_s (x_1,...., x_{t-1})\\ 
x_t < M_1(x_1,...., x_{t-1})\\
.......\\
.......\\
x_t < M_r(x_1,...., x_{t-1}) \\
N_1 (x_1,...., x_{t-1}) >0 \\
.......\\
.......\\
N_p (x_1,...., x_{t-1}) >0 \\
\end{array} \right. $$ 
for some linear polynomials $L_i, M_j, N_l$ in $x_1, ...., x_{t-1}$. The system $S$ is solvable if and only if the system in $x_1,..., x_{t-1}$  given by the inequalities 
$$ M_j (x_1,...., x_{t-1})> L_i (x_1,...., x_{t-1}),$$ for $i=1,...., s$, $j=1,...., r$,
and the inequalities 
$$N_l (x_1,...., x_{t-1}) >0 ,$$
for $l=1,....., p$,
  is solvable.  We get an analogous statement  if we replace 
some of the strict inequalities with  nonstrict inequalities.
\end{rem}

\begin{lem} \label{stnost} Let  $z_1,...., z_s,t \in \N-\{0\}$ and let 
$L_i (x_1,..., x_t)$ for $i=1,...., r$  be polynomials of degree $1$
in the unknowns $x_1,...., x_t$.
If, for any $ \varepsilon >0$, the system 
\begin{equation} \label{sysesp}
  \left\{ \begin{array}{l}
L_1 (x_1,...., x_t) > - z_1 \varepsilon \\
.......\\
.......\\
L_s (x_1,...., x_t) > - z_s \varepsilon \\
L_{s+1} (x_1,...., x_t)  \geq 0 \\
.......\\
.......\\
L_{r} (x_1,...., x_t)  \geq 0 
\end{array} \right.
\end{equation}
is solvable, then also the system 
\begin{equation} \label{sys0}
  \left\{ \begin{array}{l}
L_1 (x_1,...., x_t) \geq 0 \\
.......\\
.......\\
L_s (x_1,...., x_t) \geq 0 \\
L_{s+1} (x_1,...., x_t)  \geq 0 \\
.......\\
.......\\
L_{r} (x_1,...., x_t)  \geq 0 
\end{array} \right. 
\end{equation} 
is solvable.
\end{lem}

\begin{proof} We prove the statement by induction on $t$.

The statement in the case $t=1$ is easy to prove.
Let us prove the induction step $ t-1 \Rightarrow t$. 
Suppose that, for any $ \varepsilon >0$, the system 
 (\ref{sysesp}) is solvable; then also the system (in $x_1, ...., x_{t-1}$) 
 we get from it by ``eliminating'' the unknown $x_t$ (see Remark  \ref{eliminating}) is solvable. By induction assumption also the system we get from it by replacing $>$ with $\geq$ and putting $ \varepsilon =0$ is solvable. But this last system is exactly the system we get from (\ref{sys0}) 
by  eliminating the unknown $x_t$. So also (\ref{sys0})  is solvable.
\end{proof}

\section{The case of graphs}

\begin{thm}
Let $\{m_I\}_{I  \in {[n] \choose 2}}$ and
$\{M_I\}_{I  \in {[n] \choose 2}}$ be
two families of positive real numbers  with $m_I \leq M_I$ for any $I$. 
There exist a positive-weighted graph ${\cal G}$
and an  $n$-subset $[n]$ of  the set of its vertices such that 
$D_I ({\cal G}) \in [m_I, M_I] $ for any $I  \in {\{1,...,n\} \choose 2}$ if
and only if for any $i,j \in [n]$ with $ i \neq j$ we have
$$ m_{i,j} \leq M_{i,t_1} + M_{t_1, t_2}+  ... + M_{t_{k-1}, t_k}+ 
M_{t_k, j}$$ for any $k \in \N $ and $ t_1,..., t_k \in [n]-\{i,j\}$ with $t_{\alpha}
\neq t_{\alpha +1}$ for any $ \alpha=1,..., k-1$.

\end{thm}

\begin{proof}
$\Rightarrow$ Suppose there exist a positive-weighted graph ${\cal G}$
and an  $n$-subset $\{1,..., n\}$ of  the set of its vertices such that 
$D_I ({\cal G}) \in [m_I, M_I] $ for any $I  \in {\{1,...,n\} \choose 2}$. 
We recall that, for the $D_I ({\cal G}) $, the triangle inequalities hold, see 
Theorem \ref{Hakimi-Yau}.
Then, for any  $i,j \in [n]$ with $ i \neq j$, 
$$ \begin{array}{ll}
m_{i,j} \leq D_{i,j} ({\cal G}) & \leq  D_{i,t_1}  ({\cal G})+ D_{t_1, j}   ({\cal G}) \leq \\
& \leq............................. \leq\\ & 
\leq  D_{i,t_1}  ({\cal G})+ D_{t_1, t_2}  ({\cal G})+  ..... + D_{t_{k-1}, t_k}
 ({\cal G})+ D_{t_k, j}  ({\cal G}) \leq \\
&  \leq
  M_{i,t_1} + M_{t_1, t_2}+  ..... + M_{t_{k-1}, t_k}+ M_{t_k, j}
  \end{array}$$
for any $k \in \N $ and  $ t_1,..., t_k \in [n]-\{i,j\}$ with $t_{\alpha}
\neq t_{\alpha +1}$ for any $ \alpha=1,..., k-1$.

$\Leftarrow$ 
Let us define, for any $i,j \in [n]$ with $ i \neq j$,
$$ \tilde{M}_{i,j} : = min \{  M_{i,t_1} + M_{t_1, t_2}+  ... + M_{t_{k-1}, t_k}+ M_{t_k, j}\}_{k \in \N , \;\;t_1,..., t_k \in [n]-\{i,j\} ,\;\;\;t_{\alpha} \neq
t_{\alpha +1} \;\;\forall \alpha=1,..., k-1}.$$
It is easy to see that the $ \tilde{M}_{i,j} $ satisfy the triangle inequalities
$ \tilde{M}_{i,j} \leq  \tilde{M}_{i,o} +  \tilde{M}_{o,j} $ for any distinct $i,j,o 
\in [n]$, 
in fact:
$$ \begin{array}{l} min \{  M_{i,t_1} + M_{t_1, t_2}+  ..... + M_{t_{k-1}, t_k}+ M_{t_k, j}\}_{k \in \N , \;\;t_1,..., t_k \in [n]-\{i,j\} ,\;\;\;t_{\alpha} \neq
t_{\alpha +1} \;\;\forall \alpha=1,..., k-1} \leq \vspace*{0.3cm} \\ 
\hspace*{1cm} \leq 
M_{i,v_1} + M_{v_1, v_2}+  ..... + M_{v_{r-1}, v_r}+ M_{v_r, o}+ 
M_{o ,w_1} + M_{w_1, w_2}+ ..... + M_{w_{s-1}, w_s}+ M_{w_s, j}
\end{array}$$
for any $r,s \in \N$, $v_{\alpha} \in [n] -\{i,o\}$ for $ \alpha = 1,....., r$, $v_{\alpha} \neq v_{\alpha +1} $ for $ \alpha = 1,....., r-1$ , $w_{\alpha} \in [n] -\{o,j\}$ for $ \alpha = 1,....., s$, 
$w_{\alpha} \neq w_{\alpha +1} $ for $ \alpha = 1,....., s-1$
(consider two cases: the case where no one of the $v_{\alpha}$ and the 
$w_{\alpha}$ is in $\{i,j\}$ and the case where at least one of the $v_{\alpha}$ or  the 
$w_{\alpha}$ is in $\{i,j\}$). 
So, by Theorem \ref{Hakimi-Yau}, there exists
a positive-weighted graph ${\cal G}$ such that $ D_{i,j} ({\cal G}) =
\tilde{M}_{i,j}$ for any  $i,j \in [n]$ with $ i \neq j$.
By our assumption, we have that $\tilde{M}_{i,j}
 \geq m_{i,j}$ for any  $i,j \in [n]$ with $ i \neq j$ and obviously $\tilde{M}_{i,j} \leq M_{i,j}$ for any  $i,j \in [n]$ with $ i \neq j$, so we conclude.
\end{proof}

\section{The case of trees}

\begin{defin} Let $X$ be  a set and let $Y $ be a $4$-subset of $ X$
(a quartet).  A {\bf (quartet) split} of  $Y $  is a partition of $Y$ into two disjoint $2$-subsets. We denote the split $ \{\{a,b\} , \{c,d\}\}$ simply 
by $(a,b \;|\; c,d)$.

Let $S$ be a system (that is, a set) of splits of the quartets of $X$.

We say that $S$ is {\bf fat} if, for every quartet of $X$,
either exactly one of its splits or all its splits are in $S$.

Following \cite{Dresslibro}, Ch. 3,
we say that $S$ is {\bf transitive} if, for any distinct $a,b,c,d,e \in X$,
the following implication holds:
$$(a, b \;| \;c, d) \in S \mbox{\ and } (a, b \;| \;c, e) \in S \; \Longrightarrow \; (a, b \;| \; d, e) \in S.$$ 

Following again \cite{Dresslibro},
we say that $S$ is {\bf saturated} if, for any distinct $a_1, a_2 , b_1 , b_2 , x \in X$, the following implication holds: $$(a_1, a_2 \;| \;b_1, b_2) \in S
\Longrightarrow  \mbox{ either } (a_1, x \;| \;b_1, b_2) \in S \mbox{ or } (a_1, a_2 \;| \;b_1, x) \in S.$$ 
\end{defin}

The statement of the following lemma is similar to the characterization
of the system of the splits of the quartets coming from trees (with a slight difference in the definition of the splits of a quartet of  leaves of a tree), see \cite{Dresslibro} Thm. 3.7 and \cite{C-S}.

\begin{lem}  \label{sistug}
Let $n \in \N$, $n \geq 4$. Let $S$ be a system of 
splits of the quartets of $[n]$. Suppose $S$ is fat, transitive and saturated.
Then the linear system in the unknowns $x_I$ for $I \in {[n] \choose 2}$ given by the equations $$ x_{a,c} - x_{b,c} = x_{a,d} - x_{b,d}$$ 
for any $ (a, b\; |\; c ,d ) \in S$  has a nonzero solution.
\end{lem}

\begin{proof}
We prove the statement by induction on $n$.
If $n=4$, the statement is obvious. Let us prove the induction step.
Suppose that $D_I$ for $I \in {[n-1] \choose 2}$ solve 
 the equations $$ x_{a,c} - x_{b,c} = x_{a,d} - x_{b,d}$$ 
for any $ (a, b \;|\; c ,d ) \in S$ with $a,b,c,d \in [n-1]$ and that they are not all zero. We want to find $D_{n,i}$ for $i=1,..., n-1$ such that the $D_I$ for $I \in {[n] \choose 2}$ solve the linear
system given by all the elements of $S$.

Let us define $D_{n,1}$ at random.

Let us define $D_{n,2}$ as follows:

if there does not exist $x \in [n-1]-\{1,2\}$ such that $(n ,x \;|\; 1 ,2 ) \in S$, we define $D_{n,2}$  at random;
 
if there exists $x \in [n-1]-\{1,2\}$ such that $(n ,x \;|\; 1 ,2 ) \in S$, we set
$$ D_{n,2}:= D_{n,1} + D_{x,2} - D_{x,1};$$
 it is a good definition, in fact  if there exists $y \in [n-1]-\{x,1,2\}$ such that $(n ,y \;|\; 1 ,2 ) \in S$, we have that, by the transitivity of $S$,  $(x,y \;|\; 1 ,2 ) \in S$,
 so $$   D_{n,1} + D_{x,2} - D_{x,1}=   D_{n,1} + D_{y,2} - D_{y,1}.$$
 In an analogous way we define the other $D_{n,i}$; precisely, suppose 
  we have defined $ D_{n,1}, ......., D_{n, k-1}$ in such a way 
that $ D_{n,1}, ......., D_{n, k-1}$ and $D_{i,j}$ for $i, j \in [n-1]$ satisfy the equations induced by 
$S$ involving  $ x_{n,1}, ......., x_{n, k-1}$ and  $x_{i,j}$ for $i, j \in [n-1]$;
we define $D_{n,k}$ as follows:

if there do not exist $x \in [n-1]$ and $i \in [k-1]$ with $x \neq k,i$ and
 such that $(n ,x \;|\; k ,i ) \in S$, we define
 $D_{n,k}$  at random;

if there exist $x \in [n-1]$ and $i \in [k-1]$ with $x \neq k,i$ 
 such that $(n ,x \;|\; k ,i ) \in S$, we set
 $$D_{n,k}:= D_{n,i} + D_{x,k} -D_{x,i}. $$  
 We have to show that it is a good definition. Suppose 
 $y \in [n-1]$ and $j \in [k-1]$ with $ y \neq k,j$ are such that $(n ,y \;|\; k ,j ) \in S$; we have to 
 show that 
 \begin{equation} \label{tesi}
    D_{n,i} + D_{x,k} - D_{x,i}=   D_{n,j} + D_{y,k} - D_{y,j}.
  \end{equation}
 Since $S$ is saturated and transitive, from $(n ,x \;|\; k ,i ) \in S$, we get:
  
 either 
 \begin{equation} \label{1a} 
 (n ,y \;|\; k ,i ) \in S  \hspace*{5mm} \mbox{and}  \hspace*{5mm} (x, y \;|\; k ,i ) \in S \end{equation}
or
 \begin{equation} \label{1b} 
 (n ,x \;|\; k ,y ) \in S  \hspace*{5mm} \mbox{and}   \hspace*{5mm} (n, x \;|\; y ,i ) \in S.
 \end{equation}

From $(n ,x \;|\; k ,i ) \in S$, we get:
 
 either 
 \begin{equation} \label{2a} 
 (n ,j \;|\; k ,i ) \in S  \hspace*{5mm} \mbox{ and }  \hspace*{5mm} (x, j \;|\; k ,i ) \in S
 \end{equation}
or
 \begin{equation} \label{2b} 
 (n ,x \;|\; k ,j ) \in S  \hspace*{5mm} \mbox{ and }   \hspace*{5mm}(n, x \;|\; j ,i ) \in S.
 \end{equation}

From $(n ,y \;|\; k ,j ) \in S$, we get:
 
 either 
 \begin{equation} \label{3a} 
 (n ,x \;|\; k ,j ) \in S  \hspace*{5mm} \mbox{ and }  \hspace*{5mm} (x, y \;|\; k ,j ) \in S
 \end{equation}
or
 \begin{equation} \label{3b} 
 (n ,y \;|\; k ,x ) \in S \hspace*{5mm} \mbox{ and }  \hspace*{5mm} (n, y \;|\; x ,j) \in S.
 \end{equation}

Finally, from $(n ,y \;|\; k ,j ) \in S$, we get:
 
 either 
 \begin{equation} \label{4a} 
 (n ,i \;|\; k ,j ) \in S  \hspace*{5mm} \mbox{ and }  \hspace*{5mm} (y ,i \;|\; k ,j ) \in S
 \end{equation}
or
 \begin{equation} \label{4b} 
(n ,y \;|\; k ,i ) \in S \hspace*{5mm} \mbox{ and }  \hspace*{5mm} (n, y \;|\; i ,j) \in S.
 \end{equation}

If condition (\ref{3a}) holds, we get, from it and from the assumption 
$(n ,x \;|\; k ,i ) \in S$, that also $(n ,x \;|\; i ,j ) \in S $ holds (by the transitivity of $S$). So the statement
(\ref{tesi})  is equivalent to the equality 
$$   D_{x,i} + D_{x,k} - D_{x,i}=   D_{x,j} + D_{y,k} - D_{y,j},$$
which follows from $(x ,y\;|\; k ,j ) \in S$.

If condition (\ref{1a}) holds, we get our statement in an analogous way
(swap $i$ with $j$ and $x$ with $y$).

If condition (\ref{4b}) holds, we get, from it and from the assumption 
$(n ,x \;|\; k ,i ) \in S$, that also $(x,y \;|\; k ,i ) \in S $ holds. 
From the condition that $(n ,y \;|\; i,j ) \in S$,    the statement
(\ref{tesi})  is equivalent to the equality 
$$   D_{y,i} + D_{x,k} - D_{x,i}=   D_{y,j} + D_{y,k} - D_{y,j},$$
which follows from $(x ,y\;|\; k ,i) \in S$.

If condition (\ref{2b}) holds, we get our statement in an analogous way
(swap $i$ with $j$ and $x$ with $y$).

So we can suppose that (\ref{3b}), (\ref{1b}), (\ref{4a}), (\ref{2a}) hold.
From the fact that $(n ,j \;|\; k ,i ) \in S$ (which is true by (\ref{2a})), the fact that $(n ,i\;|\; k ,j ) \in S$ (which is true by (\ref{4a})) and the fatness of $S$, 
we get that $(n ,k \;|\; i ,j ) \in S$. From the condition that $(x ,j \;|\; k,i ) \in S$ (which is true by (\ref{2a})),   the statement (\ref{tesi})  is equivalent to the equality 
$$   D_{n,i} + D_{j,k} - D_{j,i}=   D_{n,j} + D_{y,k} - D_{y,j}.$$
By the condition  $(n ,k \;|\; i ,j ) \in S$, this equality is equivalent to 
$$   D_{k,i} + D_{j,k} - D_{j,i}=   D_{k,j} + D_{y,k} - D_{y,j}.$$
which is true since $(i,y\; |\; k,j)\in S$ (which follows from (\ref{4a})).
\end{proof}

\begin{thm} \label{treesdisstr}
Let $\{m_I\}_{I  \in {\{1,...,n\} \choose 2}}$ and
$\{M_I\}_{I  \in {\{1,...,n\} \choose 2}}$ be
two families of real numbers  with $m_I < M_I$ for any $I$.
There exists a weighted tree ${\cal T}=(T,w)$
with $L(T) = [ n]$ and such that 
$D_I ({\cal T}) \in (m_I, M_I) $ for any $I  \in {\{1,...,n\} \choose 2}$ if
and only if 
there exists a set  $S$ of splits of the quartets of $[n]$ such that 

(i) $S$ is fat, transitive and saturated,

(ii) $$ m_{\sigma_1}+.....+m_{ \sigma_r} < M_{\tau_1} +..... + M_{\tau_r}$$ 
for any $r \in \N-\{0\}$, for any $(\sigma_1,...., \sigma_r) $ and $(\tau_1,...., \tau_r)$ partitions of the same $2r$-subset of $[n]$ into $2$-sets such that $(\sigma_1,...., \sigma_r) $ can be obtained from  $(\tau_1,...., \tau_r)$ with transformations on the $2$-sets
of the following kind:
$$ (i,k) ,( j,l) \mapsto (i,j ),( k,l) $$ for any $ (j,k \; |\; i,l) \in S$. 
\end{thm}

\begin{proof}
$\Rightarrow$ 
Let  ${\cal T}=(T,w)$ be a  weighted tree
with $L(T) = [ n]$ and such that 
$D_I ({\cal T}) \in (m_I, M_I) $ for any $I  \in {\{1,...,n\} \choose 2}$. 
We define $S$ in the following way: for any quartet $\{a,b,c,d\}$ in $[n]$,
we say that $(a, b\; |\; c,d ) \in S$ if and only if $ a$ and $b$ are neighbours 
and $c$ and $d$ are neighbours in $ T|_{a,b,c,d}$. It is easy to see that $S$ 
is fat, transitive and saturated. Furthermore,  
for any $(\sigma_1,...., \sigma_r) $ and $(\tau_1,...., \tau_r)$ partitions of the same subset of $[n]$ into $2$-sets such that $(\sigma_1,...., \sigma_r) $ can be obtained from  $(\tau_1,...., \tau_r)$ with transformations on the $2$-sets
of the  kind
$ (i,k \;|\; j,l) \mapsto (i,j \;| k,l) $ for any $ (j,k \; |\; i,l) \in S$, we have: 
$$ m_{\sigma_1}+.....+m_{ \sigma_r} < D_{\sigma_1} ({\cal T}) + ......+ D_{\sigma_r} ({\cal T}) = D_{\tau_1} ({\cal T}) + ......+ D_{\tau_r} ({\cal T}) <
  M_{\tau_1} +..... + M_{\tau_r},$$ 
hence (ii) holds.

$\Leftarrow$ 
By Lemma \ref{sistug}, the linear system given by the
equations $$ D_{a,c} - D_{b,c} = D_{a,d} - D_{b,d}$$ 
for any $ (a, b\; |\; c ,d ) \in S$ has nonzero solutions. So we can write some unknowns, 
$D_{I_1},....., D_{I_s}$,
 in function of some others: $D_{J_1},   .........,  D_{J_t}$ for some $t \geq 1$:
let $$D_{I_i} = f_{I_i} (D_{J_1},........, D_{J_t})$$ for $i=1,...,s$.
    Consider the following system of inequalities in $D_{J_1},........, D_{J_t}$:
   \begin{equation} \label{sys}
   \left\{ \begin{array}{l}
    D_{J_1} - m_{J_1} >0  \\
      ........\\ 
        ........\\ 
       D_{J_t} - m_{J_t} >0  \\
    f_{I_1} (D_{J_1},........, D_{J_t}) -m_{I_1} >0\\
      ........\\ 
        ........\\ 
          f_{I_s} (D_{J_1},........, D_{J_t}) -m_{I_s} >0 \\
            -D_{J_1} +M_{J_1} >0  \\
    ........\\ 
      ........\\  
       -D_{J_t} +M_{J_t} >0  \\
    -f_{I_1} (D_{J_1},........, D_{J_t}) +M_{I_1} >0\\
      ........\\ 
      ........\\  
         - f_{I_s} (D_{J_1},........, D_{J_t}) +M_{I_s} >0  
\end{array}
\right.
\end{equation}
By condition (ii)  there does not exist a set of $ 2t + 2s +1 $ nonnegative constants, $c_1,....., c_{2t + 2s+1}$, with at least one of them positive, such that the linear combination of the first members of the
inequalities of (\ref{sys})  with coefficients $c_1,....., c_{2t + 2s}$
 plus $c_{2t + 2s+1}$ is identically zero.
So, by Carver's Theorem, the system (\ref{sys}) is solvable.
Let $(\overline{D}_{J_1},   ........., \overline{ D}_{J_t})$ be a solution.
By the fatness of $S$ and by  Theorem \ref{Bandelt-Steel}, for the dissimilarity vector with entries $$\overline{D}_{J_1},   ........., \overline{  D}_{J_t},
 \overline{  D}_{I_1} := f_{I_1} ( \overline{  D}_{J_1},........, \overline{   D}_{J_t}),.....,  \overline{  D}_{I_s} := f_{I_s} ( \overline{  D}_{J_1},........, \overline{   D}_{J_t}),$$
 there exists a weighted tree ${\cal T}=(T,w)$ with $L(T) = [ n]$ and such that 
$D_I ({\cal T}) =  \overline{  D}_I$ for any $I  \in {\{1,...,n\} \choose 2}$, so 
$D_I ({\cal T}) \in (m_I, M_I) $ for any $I  \in {\{1,...,n\} \choose 2}$.
\end{proof}

\begin{rem} Observe that the same technique can  be useful to study the analogous problem for some kind of tree. For instance we can prove 
easily in an analogous way that, given two families of real  numbers,
 $\{m_I\}_{I  \in {\{1,...,n\} \choose 2}}$ and
$\{M_I\}_{I  \in {\{1,...,n\} \choose 2}}$  with $m_I < M_I$ for any $I$,
there exists a weighted star ${\cal T}=(T,w)$
with $L(T) = [ n]$ and such that 
$D_I ({\cal T}) \in (m_I, M_I) $ for any $I  \in {\{1,...,n\} \choose 2}$ if
and only if 
 $$ m_{\sigma_1}+.....+m_{ \sigma_r} < M_{\tau_1} +..... + M_{\tau_r}$$ 
for any $(\sigma_1,...., \sigma_r) $ and $(\tau_1,...., \tau_r)$ partitions of the same subset of $[n]$ into $2$-sets.
\end{rem}

Considering $2$-weights in closed intervals,
 we get the following theorem.

\begin{thm} \label{treesnost}
Let $\{m_I\}_{I  \in {\{1,...,n\} \choose 2}}$ and
$\{M_I\}_{I  \in {\{1,...,n\} \choose 2}}$ be
two families of  real numbers  with $m_I \leq M_I$ for any $I$.  
There exists a weighted tree ${\cal T}=(T,w)$
with $L(T) = [ n]$ and such that 
$D_I ({\cal T}) \in [m_I, M_I] $ for any $I  \in {\{1,...,n\} \choose 2}$ if
and only if 
there exists a system  $S$ of splits of the quartets of $[n]$ such that 

(i) $S$ is fat, transitive and saturated,

(ii) $$ m_{\sigma_1}+.....+m_{ \sigma_r} \leq M_{\tau_1} +..... + M_{\tau_r}$$ 
for any $(\sigma_1,...., \sigma_r) $ and $(\tau_1,...., \tau_r)$ partitions of the same subset of $[n]$ into $2$-sets such that $(\sigma_1,...., \sigma_r) $ can be obtained from  $(\tau_1,...., \tau_r)$ with transformations on the $2$-sets
of the following kind:
$$ (i,k) ,( j,l) \mapsto (i,j) ,(k,l) $$ for any $ (j,k \;|\;  i,l) \in S$. 
\end{thm}

\begin{proof}
The proof of the implication $\Rightarrow $ is completely analogous to the
proof of the same implication of Theorem  \ref{treesdisstr}.
Let us prove the other implication. 
By Lemma \ref{sistug}, the linear system given by the
equations $ D_{a,c} - D_{b,c} = D_{a,d} - D_{b,d}$ 
for any $ (a, b\; |\; c ,d ) \in S$  has nonzero solutions. So we can write some unknowns, 
$D_{I_1},....., D_{I_r}$,
 in function of some others $D_{J_1},   .........,  D_{J_t}$ for some $t \geq 1$:
let $$D_{I_i} = f_{I_i} (D_{J_1},........, D_{J_t})$$
for any $i=1,....,r$.
    Consider the system of inequalities 
   \begin{equation} \label{syseps}
   \left\{ \begin{array}{l}
    D_{J_1} - m_{J_1} +\epsilon >0 \\
      ........\\ 
        ........\\ 
       D_{J_t} - m_{J_t} + \epsilon >0  \\
    f_{I_1} (D_{J_1},........, D_{J_t}) -m_{I_1} + \epsilon>0\\
      ........\\ 
        ........\\ 
          f_{I_s} (D_{J_1},........, D_{J_t}) -m_{I_s} +\epsilon >0 \\
            -D_{J_1} +M_{J_1} + \epsilon  >0\\
      ........\\ 
        ........\\  
       -D_{J_t} +M_{J_t} +\epsilon >0  \\
    -f_{I_1} (D_{J_1},........, D_{J_t}) +M_{I_1} + \epsilon >0\\
      ........\\ 
        ........\\  
         - f_{I_s} (D_{J_1},........, D_{J_t}) +M_{I_s} + \epsilon >0  
\end{array}
\right.
\end{equation}

By condition (ii), we have that, for any
 $ \epsilon >0$,
$$ m_{\sigma_1}+.....+m_{ \sigma_r} - (2r)\epsilon < M_{\tau_1} +..... + M_{\tau_r}$$ for any $r \in \N-\{0\}$,
for any $(\sigma_1,...., \sigma_r) $ and $(\tau_1,...., \tau_r)$ partitions of the same $2r$-subset of $[n]$ into $2$-sets such that $(\sigma_1,...., \sigma_r) $ can be obtained from  $(\tau_1,...., \tau_r)$ with transformations on the $2$-sets
of the kind $ (i,k \;|\; j,l) \mapsto (i,j \;| k,l) $ for any $ (j,k \; |\; i,l) \in S$. 
So  there does not exist a set of $ 2t + 2s +1 $ nonnegative constants, $c_1,....., c_{2t + 2s+1}$ with at least   one of them positive, such that the linear combination of the first members of the
inequalities of (\ref{syseps})  with coefficients $c_1,....., c_{2t + 2s}$
 plus $c_{2t + 2s+1}$ is identically zero.
So, by Carver's theorem, the system (\ref{syseps}) is solvable  for any
 $ \epsilon >0$.
Hence, by Lemma \ref{stnost}, the system we get from (\ref{syseps})
by replacing $ \geq $ with $>$  and $\epsilon $ with $0$ is solvable.
Let $(\overline{D}_{J_1},   .........,\overline{  D}_{J_t})$ be a solution.

By the fatness of $S$ and by  Theorem \ref{Bandelt-Steel}, for the dissimilarity vector with entries 
$$\overline{D}_{J_1},   ........., \overline{  D}_{J_t},
 \overline{  D}_{I_1} := f_{I_1} ( \overline{  D}_{J_1},........, \overline{   D}_{J_t}),.....,  \overline{  D}_{I_s} := f_{I_s} ( \overline{  D}_{J_1},........, \overline{   D}_{J_t}),$$
 there exists a weighted tree ${\cal T}=(T,w)$ with $L(T) = [ n]$ and such that 
$D_I ({\cal T}) = \overline{D}_I$ for any $I  \in {\{1,...,n\} \choose 2}$, so 
$D_I ({\cal T}) \in [m_I, M_I] $ for any $I  \in {\{1,...,n\} \choose 2}$.

\end{proof}

By using Theorem \ref{Bune}, we get  a  theorem, analogous to the previous ones, for positive-weighted trees:

\begin{thm}
Let $\{m_I\}_{I  \in {\{1,...,n\} \choose 2}}$ and
$\{M_I\}_{I  \in {\{1,...,n\} \choose 2}}$ be
two families of positive real numbers  with $m_I < M_I$ for any $I$.  
There exists a positive-weighted tree ${\cal T}=(T,w)$
with $L(T) = [ n]$ and such that 
$D_I ({\cal T}) \in (m_I, M_I) $ for any $I  \in {\{1,...,n\} \choose 2}$ if
and only if 
there exists a system  $S$ of splits of the quartets of $[n]$ such that 
the  condition (i)  of Theorem \ref{treesdisstr} and the following condition hold:

(ii) $$ m_{\sigma_1}+.....+m_{ \sigma_r} < M_{\tau_1} +..... + M_{\tau_s}$$ 
for any $(\sigma_1,...., \sigma_r) $ and $(\tau_1,...., \tau_s)$ partitions of  subsets of $[n]$ into $2$-sets such that $(\tau_1,...., \tau_s)$  can be obtained from  $(\sigma_1,...., \sigma_r) $
 with transformations on the $2$-sets of the following kind:

$$ (i,k) ,( j,l) \mapsto (i,j) ,( k,l) $$ 
for any $ (j,k| i,l) \in S$,
$$ (i,k) ,( j,l) \mapsto (i,j) ,( k,l) $$ 
for any $ (i,k| j,l) \in S$ such that $(i,k| j,l)$ is the only split of $\{i,j,k,l\}$
 in $S$,
   $$ (a,b) \mapsto (a,c) , (c,b)$$ 
for any $a,b,c \in [n]$.
\end{thm}

The proof is very similar to the one of Theorem \ref{treesdisstr}; the only difference is that in the system (\ref{sys}) we have to consider also the inequalities induced (by replacing the $D_{I_i}$ with the $f_{I_i}$) by the 
 inequalities $$ D_{a,c} + D_{c,b}  - D_{a,b} + \varepsilon >0$$ 
for any distinct $a,b,c \in [n]$ and the inequalities 
$$ D_{a,b} + D_{c,d}  < D_{a,c} + D_{b,d}$$ for 
any quartet $\{a,b,c,d\}$ in $[n]$ such that there is only one of its 
splits, $ (a,b \; | \; c,d) $, in $S$.

\bigskip

{\bf Address:}
Dipartimento di Matematica e Informatica ``U. Dini'', 
viale Morgagni 67/A,
50134  Firenze, Italia

{\bf
E-mail address:}
rubei@math.unifi.it

\end{document}